\documentclass[11pt,letterpaper]{amsart}
\usepackage{graphicx,amsmath,hyperref} 

\usepackage{fullpage, amsmath,amssymb,amsthm,amsfonts,enumerate,textcomp, eurosym, epsfig,epstopdf, tikz, esint, bigints,mathrsfs,todonotes,mathtools,verbatim,bm,nccmath, dirtytalk}
\input xy
\xyoption{all}

\begin{document}

\title{Quantum $K$-Rings of Partial Flag Varieties, Coulomb Branches, and the Bethe Ansatz}
\author{Irit Huq-Kuruvilla}
\address{Department of Mathematics\\
Virginia Tech\\
Blacksburg, VA 24060}
\email{irithk@vt.edu}
\date{\today}

\theoremstyle{definition}
\newtheorem{mydef}{Definition}[section]
\newtheorem{ex}{Example}[section]
\theoremstyle{plain}
\newtheorem{thm}{Theorem}[section]

\newtheorem*{thm*}{Theorem}
\newtheorem{lem}[thm]{Lemma}
\newtheorem{conj}{Conjecture}[section]
\newtheorem{prop}[thm]{Proposition}
\newtheorem{cor}[thm]{Corollary}
\newtheorem{exercise}[thm]{Exercise}
\newtheorem{remark}[thm]{Remark}
\begin{abstract}
We give a purely geometric explanation of the coincidence between the Coulomb Branch equations for the 3D GLSM describing the quantum $K$-theory of a flag variety, and the Bethe Ansatz equations of the 5-vertex lattice model. In doing so, we prove two explicit presentations for the quantum $K$-ring of the flag variety, resolving  conjectures of Gu-Sharpe-Mihalcea-Xu-Zhang-Zou and Rimanyi-Tarasov-Varchenko. We also prove that the stable map and PSZ quasimap $K$-theory of the partial flag varieties are isomorphic, using the work of Koroteev-Pushkar-Smirnov-Zeitlin identifying the latter ring with the Bethe algebra of the 5-vertex lattice model. Our isomorphism gives a more explicit description of the quantum tautological bundles described in the PSZ quasimap ring.

\end{abstract}
\maketitle
\section{Introduction}
The quantum $K$-ring of a smooth projective variety $X$, introduced by Givental and Lee, is a deformation of the usual $K$-ring of $X$ using $K$-theoretic Gromov-Witten of invariants of $X$. It is a generalization to $K$-theory of the quantum cohomology ring.

While there are many occurrences of the quantum $K$-ring in physics, integrable systems, and reprentation theory, actually fully describing the ring is generally much more difficult than the quantum cohomology ring, so fewer computations have been done. 

In particular, there is no proven description of the ring for $X=Fl(v_1,\dots,v_k;N)$, the partial flag varieties. $Fl(v_1,\dots,v_k;N)$ is the moduli space of flags of subspaces $V_1\subset V_2\subset \dots V_k\subset \mathbb{C}^N$, with $dim(V_i)=v_i$. 

The subspace $V_i$ induces a tautological bundle $\mathcal{S}_i$ over $X$, we use the convention that $\mathcal{S}_{i+1}=\mathbb{C}^N$. The standard torus action on $\mathbb{C}^N$ induces an action of the flag variety. Let $\mathcal{R}_i$ denote the tautological quotient bundle $\mathcal{S}_{i+1}/\mathcal{S}_i$. 

The exterior powers of $\mathcal{S}_i$ generate the ($T$-equivariant or not) $K$-ring of the variety, with a set of relations determined entirely by the Whitney sum formula:

\begin{equation*}
\Lambda_y(\mathcal{S}_{i+1})=\Lambda_y(\mathcal{S}_i)\Lambda_y(\mathcal{R}_i)
\end{equation*}

There are two sets of predictions for $QK(Fl)$, with remarkable interplay between them. One, due to Gu-Mihalcea-Sharpe-Xu-Zhang-Zhou in \cite{conjpf} (first introduced in \cite{conj} for Grassmanians), which is related to the OPE ring of a certain 3D gauged linear sigma model, and gives the following conjectural description of the relations of the $T$-equivariant quantum $K$-ring of the flag, henceforth referred to as the \emph{Whitney presentation}:
\begin{conj}
    \label{whitconj}

\begin{equation}
\Lambda_y(\mathcal{S}_i) * \Lambda_y( \mathcal{R}_i)
\: = \:
\Lambda_y( \mathcal{S}_{i+1}) \: - \:
y^{v_{i+1}-v_i} \frac{Q_i}{1-Q_i} \det(\mathcal{R}_i) *
\left( \Lambda_y(\mathcal{S}_i) - \Lambda_y(\mathcal{S}_{i-1}) \right).
\end{equation}
\end{conj}
 The special cases of Grassmanians and incidence varieties by a subset of the same authors have been addressed in \cite{conj} and \cite{qwhit}.

This conjecture was obtained via symmetrizing the Coulomb branch equations for the GLSM, which are given by the critical locus of a superpotential $\mathcal{W}$, after choosing a particular set of Chern-Simons levels. 

 There is one equation for each pair $(i,j)$ with $P^i_{j}$ being the $j$th Chern root of $\mathcal{S}_i$:

\begin{equation*}
(-1)^{v_i-1}\prod_k P^i_k\prod_{b=1}^{v_{n+1}}\prod_k(1-\frac{P^i_j}{P^{i+1}_b})=(P^i_j)^{v_i}Q_{ij}\prod_{a=1}^{v_{n-1}}(1-\frac{P^{i-1}_a}{P^{i}_j})=0
\end{equation*}

The other prediction comes from integrable systems. A general conjecture due to  Rimanyi-Tarasov-Varchenko in \cite{rimanyi2015trigonometric} conjectured that the quantum $K$-theory of the Nakajima quiver variety $T^*Fl$ was isomorphic to the Bethe algebra of the Yang-Baxter algebra associated to a particular quantum group, the Yangian. The same Yang-Baxter algebra also arises using from Bethe Ansatz method to construct solutions to the Yang-Baxter equation for the asymmetric 6-vertex lattice model, a quantum integrable system. The \say{compact limit} of this algebra was predicted to describe quantum $K$-theory of the flags themselves, and corresponds to the Bethe algebra of the 5-vertex model, a more degenerate integrable system. Based on this isomorphism, Rimanyi-Tarasov-Varchenko in Conjecture 13.17 of \cite{rimanyi2015trigonometric} give the following prediction for the quantum $K$-ring of the partial flag:

\begin{conj}
\label{rtva}
The quantum $K$-ring has a presentation given by the determinant of a discrete Wronskian matrix $W_Q$. 
\end{conj}

Rimanyi-Tarasov-Varchenko's question was partially answered by Koroteev-Pushkar-Smirnov-Zeitlin in \cite{manybody}. Those authors defined a variant of quantum $K$-theory based on the quasimap moduli space (here denoted $QK^{QM}$), rather than the stable map moduli space, that realized the isomorphism to the Bethe algebra. We refer to their theory as the PSZ quasimap theory to disambiguate from the work of Ruan-Zhang and Tseng-You. They identified the spectra of operators of multiplication by certain classes in $QK^{QM}(T^*Fl)$ with solutions to the Bethe Ansatz equations of the Yangian. 

 For cotangent bundles, based on the work of \cite{Liu} and \cite{Xiaohan}, the quasimap rings are not expected to coincide with the stable map rings due to differences between the corresponding generating functions.

However, in the compact limit, those functions coincide making it reasonable to expect the following: 

\begin{conj}
    \label{iso}
    $QK(Fl)\cong QK^{QM}(Fl)$
\end{conj}
This prediction was verified in \cite{manybody} for the specific case of full flag $SL(N)/B$, where the authors used it to prove Conjecture \ref{rtva} in that case. 

The explicit description of $QK^{QM}(Fl)$ in \cite{manybody} is as follows:

\begin{thm*}[Koroteev-Pushkar-Smirnov-Zeitlin]
    Let $\tau$ be a function of the Chern roots $P^i_j$ of $\mathcal{S}_i$, invariant under the action of $\prod_i S_{v_i}$. The authors define $\widehat{\tau}$ a certain deformation of the $K$-theory class $\tau(P^i_j)$. Their theorem shows that the eigenvalues of $\widehat{\tau}$ against the class of a $T$-fixed point are given by the Bethe Ansatz equations, which are:

\begin{equation}
(-1)^{v_i-1}\prod_k P^i_k\prod_{b=1}^{v_{n+1}}\prod_k(1-\frac{P^i_j}{P^{i+1}_b})=(P^i_j)^{v_i}Q_{ij}\prod_{a=1}^{v_{n-1}}(1-\frac{P^{i-1}_a}{P^{i}_j})=0
\end{equation}

We note that the $Q$-deformation used to define $\widehat{\tau}$ are difficult to compute in practice, and the interpretation remains somewhat mysterious. 
\end{thm*}

In a remarkable coincidence, the Bethe Ansatz equations are identical to the Coulomb branch equations. From a physical perspective, this coincidence can be thought of as the compact limit of the gauge/Bethe correspondence for $T^*Fl$, formulated by Nekrasov-Shatashvili in \cite{ns} for general Nakajima quiver varieties. 

We propose a purely geometric explanation of this coincidence by providing a direct interpretation of the Bethe Ansatz equations themselves in terms of quantum $K$-theory, and in doing so prove both predicted descriptions of $QK(Fl)$. To do this, we use the quantum $K$-theoretic abelian/non-abelian correspondence, formulated in general by the author in \cite{me}.
\begin{conj}
\label{abnonab}
For a GIT quotient $V//G$, under some mild conditions, there 
there is a surjection $\phi_Q: QK^{tw}(V//T_G)^W\to QK(V//G)$, where $T_G$ denotes the maximal torus of $G$, and $W$ is the Weyl group of $G$. The superscript $tw$ denotes a twisting, i.e. a modification of the virtual structure sheaf. We refer to $V//T_G$ as the \emph{abelianization} of $V//G$.
\end{conj}

In the same work, the author gave a proof of this conjecture for the standard GIT description on $Fl$, given by
$$Hom(\mathbb{C}^{v_1},\mathbb{C}^{v_{2}})\times\dots Hom(\mathbb{C}^{v_n},\mathbb{C}^N)//\prod_{i=1}^n GL(v_i)$$

In this work, we use the above result to compute the ring $QK(Fl)$. In doing so, we discover the following: 

Let $Y$ denotes the abelianization of $Fl$. It turns out that the natural interpretation of the Bethe Ansatz/Coulomb Branch equations are as limits of relations in $QK^{tw}(Y):$
\begin{thm}
    
For certain choices of bundles $P^i_j$, identified with the Chern roots of $\mathcal{S}_i$ by $\phi$, we have that $QK^{tw}(Y)$ is determined by the following relations, for each $i,j$. 

$$\prod_{b=1}^{v_{n+1}}\prod_k(1-\frac{P^i_j}{P^{i+1}_b})=Q_{ij}\prod_{a=1}^{v_{n-1}}(1-\frac{P^{i-1}_a}{P^{i}_j})\prod_{k\neq j}\frac{(1-\lambda\frac{P_j^i}{P_k^{i}})}{ (1-\lambda \frac{P^i_k}{P^i_j})}$$

\end{thm}

After setting $\lambda$ to 1, and specializing $Q^i_j$ to $Q_i$, we recover the Bethe Ansatz. This provides a direct geometric interpretation for the Bethe Ansatz equation.

Using this result, we prove the following theorems:

\begin{thm}
\hfill
\begin{enumerate}
    \item Conjecture \ref{whitconj} is true. (We actually prove a stronger version, given in equation \ref{gamer}, this conjecture appears in \cite{conj} but not \cite{incidence}.)
    
    \item Conjecture \ref{rtva} is true, after some quantum corrections which can be calculated explicitly. These corrections are implicit in the work of \cite{manybody} where study the case of the full flag.

   \item Conjecture \ref{iso} is true, with the isomorphism $QK^{QM}(Fl)\to QK(Fl)$ given by $\widehat{\tau}\mapsto \phi_Q(\tau(P^i_j))$.
\end{enumerate}
\end{thm}

\subsection*{Connection to other work}
 The work of Koroteev-Pushkar-Smirnov-Zeitlin in \cite{manybody} identifies the quasimap ring with the Bethe algebra by means of Baxter's $Q$-operator, which arises in quantum $K$-theory as a generating function for the operation of quantum multiplication by the quantum tautological bundles. This is sufficient to establish an isomorphism of the collection of quantum $K$-rings for a given $N$ with the Bethe algebra. 

 The work Korff-Gorbunov in \cite{korff} prove a more explicit identification for the non-equivariant (stable map) quantum $K$-ring of the Grassmanian, that in particular computes the Bethe eigenvectors, and the transfer matrices of the algebra, which are not identified explicitly in \cite{manybody}. Their work also uses the Bethe Ansatz equation, but does not regard Chern roots of tautological bundles as solutions to the equation. We do not know of an explicit relationship ebtween our work and theirs. 

The relations we obtain after passing to the abelianization are symbols of $q$-difference operators acting on certain $J$-functions. The author was recently made aware that the connection between these operators and the Bethe Ansatz was observed by W. Gu in \cite{conj} (see remark 9.5). One can regard Conjecture 1.2 (originally made in \cite{me}), as the precise statement of the abelian/non-abelian correspondence conjecture suggested in that remark.

This project is contemporaneous with one joint with Amini-Mihalcea-Orr-Xu that shows that the Whitney presentation for full flags implies Conjecture \ref{whitconj} in general, using Kato's pushforward theorem for $QK(G/P)$.

\subsection*{Acknowledgements}
The author thanks Kamyar Amini, Leonardo Mihalcea, Daniel Orr, and Weihong Xu for useful discussions. 

The author is thankful to Wei Gu, Eric Sharpe, and Hao Zou for explaining the physical ideas used in the theory. 

The author is thankful to Christian Korff and Mikhail Vassiliev for explaining the relevant integrable systems, and for welcoming him to Glasgow, where his visit was funded by the Royal Society International Exchanges grant \emph{Quantum K-theory and quantum integrable models}.  

The author also wishes to thank A. Smirnov for explanations of his work.
As well as X. Yan for some useful corrections. 

\section{Quantum $K$-Rings and the $J$-function}

\subsection{$K$-theoretic Gromov-Witten Invariants}
$K$-theoretic Gromov-Witten invariants are defined as certain holomorphic Euler characteristics on the Kontsevich moduli space of stable maps $X_{g,n,d}(X)$ (this will be shortened to $X_{g,n,d}$ for the remainder of this text). Given Laurent polynomials $f_i(q)$ with coefficients in $K^*(X)$, the associated Gromov-Witten invariant is written in correlator notation as $\langle f_1,\dots,f_n\rangle_{g,n,d}$ and is defined to be:

$$\chi(X_{g,n,d};\mathcal{O}^{vir}\otimes \prod_i ev_i^*f_i(L_i))$$

Here the notation $ev_i^*f_i(L_i)$ means to pull back the coefficients of $f_i$ by $ev_i$, and evaluate the resulting polynomial at $q=L_i$.
\subsubsection{Twistings}

Given a vector bundle $V$ and an (invertible) $K$-theoretic characteristic class $C$, we can define \emph{twisted} $K$-theoretic Gromov-Witten invariants by replacing $\mathcal{O}_{vir}$ with $\mathcal{O}_{vir}\otimes C(\pi_*ev_{n+1}^*V)$. \\

This also works if $C$ is not invertible by introducing some equivariant parameter and then passing to the limit, so we are allowed to let $C$ be the $K$-theoretic Euler class. 

\subsection{Quantum $K$-ring}

The genus-zero 2 and 3 pointed invariants of a target $X$ are organized into the (small) quantum $K$-ring, a deformation of of $K^*(X)$ over the Novikov variables $Q_\alpha$, the semigroup ring of the effective curve classes in $H_2(X)$. The ring is defined as follows. 

First, define the quantum pairing, by:

$$((a,b)):=\sum_d Q^d \langle a,b\rangle_{0,2,d}$$

The product in the quantum $K$-ring is defined as a deformation of the structure constants of $K^*(X)[[Q]]$, with respect to the quantum pairing, in the following way. 

$$((a*b,c)):=\sum_{d}Q^d\langle a,b,c\rangle_{0,3,d}$$

\begin{thm}[Y-P Lee]
These operations make $QK(X)$ a commutative associative ring with unit 1. 
\end{thm}

One can define the same notions using twisted invariants, to define an object $QK^{tw}(X)$, which is a different deformation of $K^*(X)$, depending on the choice of twisting. In \cite{me}, it was established that doing this also results in a ring:

\begin{thm}
$QK^{tw}(X)$ is a commutative associative ring with unit 1. 
\end{thm}

\begin{ex}
    For $X=\mathbb{P}^n$, the Novikov ring is generated by a single variable $Q$, corresponding to the dual to the hyperplane class. The quantum $K$-ring is determined by the relation:

    $$(1-\mathcal{O}(-1))^{n+1}=Q$$

    We can also begin by using the equivariant $K$-theory. For $T$ the standard $n+1$-torus action on $\mathbb{P}^n$, with parameters $\Lambda_i$, the equivariant $K$-theory $QK_T(\mathbb{P}^n)$ is determined by the relations:

    $$\prod_i(1-\mathcal{O}(-1)\Lambda_i^{-1})=Q$$
    \end{ex}
\subsection{$J$-Function}

Given a target space $X$ and $\phi_\alpha$ a basis for $K^*(X)$, we define the (small) $J$-function of $X$ as :

$$J_X(q)=1-q+\sum_{d>0,n} \langle\phi^\alpha\frac{\phi_\alpha}{1-qL_1}\rangle_{0,1,d} $$

The $J$-function is related to the quantum $K$-ring in the following way. If we can choose a set of line bundles $P_i$ with $-c_1(P_i)$ corresponding to the Novikov variable $Q_i$, we have the following theorem due to Iritani-Milanov-Tonita

\begin{thm}[Iritani-Milanov-Tonita \cite{IMT}]
\label{imt}
If some operator $F(q^{Q_i\partial_{Q_i}},q,Q)$ annihilates $\tilde{J}$, then: 

$$F(A_i^{com},1,Q)=0\in QK^*(X)$$
Where $A_{i,com}$ is a deformation of the operator of quantum multiplication by $P_i$. (See \cite{IMT} for more details about the construction of this deformation)
\end{thm}

Due to the presence of $A_{i,com}$, it is difficult to apply this theorem as stated in order to product explicit relations. However, in many special cases, the deformation is trivial and $A_{i,com}=P_i$. These cases are described by what was called in \cite{me} the \textit{quantum triviality theorem}, originally due to Anderson-Chen-Tseng-Iritani in \cite{act}:

\begin{thm}[Anderson-Chen-Tseng-Iritani \cite{act}]
\label{qtriv}
    Let $f(P_iq^{Q_i\partial_{Q_i}})$ be a polynomial $q-$difference operator. If $f\frac{J_X}{1-q}$ vanishes at $q=\infty$ (aside from the $Q^0$ term), then $f(A_{i,com})$ is equivalent to quantum multiplication by $f(P_i)$.
\end{thm}
This theorem has two important corollaries for computing relations in quantum $K$-rings. 

\begin{cor}
    If we denote the degree $Q^d$ term of $J_X/(1-q)$ by $J_d$, we have that if $deg(J_d)<-d_i$, then $A_{i,com}=P_i$. (This corresponds to choosing $f$ to be the identity)
\end{cor}

\begin{cor}
    If the previous hypothesis is satisfied, and $f(x_i)$ is a monomial with degree $c_i$ in $x_i$, then if $deg(J_d)<\sum c_id_i$, $f(P_i)\in QK(X)$ is equal to its classical counterpart.  
\end{cor}


Note that these theorems have no direct application in our case, since the line bundles $\det(S_{i})$ do not generate the entire $K$-ring of $Fl$. It may be possible to construct appropriate operators after localization in equivariant cohomology, however we have been unable to do so. 

Instead, we use the following theorems with two modifications, proven by the author in \cite{me}.

\begin{thm}[\cite{me}]
    Both Theorem \ref{imt} and Theorem \ref{qtriv} hold for twisted quantum $K$-rings. 
\end{thm}

\begin{thm}[\cite{me}]
    In the previous two theorems, if $J^{tw}_X$ is replaced with $J^{tw}_X+\epsilon(q)$, where $\epsilon(q)$ is a rational function with coefficients in some submodule $I$ of $K(X)\otimes \Lambda$, Theorem \ref{imt} and Theorem \ref{qtriv} both apply, but all statements about $QK^{tw}(X)$ hold up to elements of $I$. 
\end{thm}

\section{Relations vs Presentations}
The results described in the previous section give methods of establishing relations in quantum $K$-rings. However, to establish a presentation of a ring, it is not sufficient to establish that the desired relations hold, but that they they determine the ring completely. 

 By a theorem of Gu-Mihalcea-Sharpe-Xu-Zhang-Zhou \cite[Thm 4.1]{nakayama}, if the Whitney relations hold, they determine a presentation of $QK(Fl)$. This is a consequence of general commutative algebra result \cite[Thm 2.7]{nakayama}, which also implies the following:

\begin{thm}
A ring map $QK(X)\to QK^{QM}(X)$ that is an isomorphism modulo Novikov variables is an isomorphism. 
    
\end{thm}
 
\begin{proof}
    A direct consequence of \cite[Thm 2.7]{nakayama}. 
\end{proof}







\section{The Classical Abelian/non-Abelian Correspondence}

 Let $V$ be a quasiprojective variety, with a linearized action reductive group $G$. We require that $G$-semistable points are stable, and stabilizer subgroups of all such points are trivial. As consequence, the GIT quotient $V//G$ is a smooth projective variety. In the interest of appropriately crediting the authors, the theorems below also hold in some more general settings, but we restrict to this one for simplicity, since we only consider smooth projective targets. \\

 In this situation, we have a surjective Kirwan map, $k_G:K_G^*(V)\to K^*(V//G)$, which corresponds to descent of $G$-equivariant bundles to the quotient, and gives a convenient way of describing $K$-theory classes on $V//G$. In particular, beginning with a trivial bundle on $V$ with a $G$-representation, this descends to a bundle on $V//G$ of the same rank. The Kirwan map respects this descent. 

For $T_G$ a maximal torus in $G$, we can also consider the GIT quotient $V//T_G$ (note that there may be more $T_G$-stable points than $G$-stable ones). The Weyl group $W$ has a natural action on $V//T_G$, and thus $K^*(V//T_G)$. This action makes the Kirwan map $k_T$ $W$-equivariant. Each root $\alpha$ of $G$ determines a character of $T_G$, whose action on $V\times \mathbb{C}$ determines a equivariant bundle on $V$, and thus a line bundle $L_\alpha$ on $V//T_G$. \\

 The abelian/non-abelian correspondence for $K$-theory relates $K^*(X//G)$ and $K^*(X//T_G)$ in the following way:

\begin{thm}[Harada-Landweber, \cite{harada}]
\hfill
\begin{itemize}
\item $K^*(V//G)\cong \frac{K^*(V//T_G)^W}{ann(Eu(\bigoplus_a L_a))}$. We refer to this isomorphism by $\phi$
\item For $\alpha\in K^*(V//T_G)$, $\frac{1}{|W|}\chi(V//T_G;Eu(\bigoplus_a L_a)\alpha)=\chi(V//T_G; sp(\alpha))$.
\item $\phi$ makes the following diagram commute:
$
\xymatrix{
K^*_G(V)\ar[r]^{res}\ar[d]^{k_G} &K^*_{T_G}(V)^W\ar[d]^{k_{T_G}}\\ K^*(X//G) & K^*(X//T_G)\ar[l]^\phi}$

Where $res$ denotes equivariant restriction. 
\end{itemize}

\end{thm}

The meaning of the final bullet point is that any $G$-bundle $E_G$ on $X$ determines a class in $K^*(X//G)$, but also a $T_G-bundle$ $E_{T_G}$, and thus a class in $K(X//T_G)$. $\phi$ maps the latter class to the former. 

$k_G$, by virtue of being a Kirwan map, is surjective. However $k_{T_G}$ a priori may not also be surjective, since we have restricted to $W$-invariants both the source and the target. However, since we work over $\mathbb{Q}$, we can average non $W$-invariant preimages to produce a $W$-invariant one. So we can treat all elements of $K^*(X//T_G),K^*(X//G)$ as coming from equivariant classes on $X$. Harada-Landweber's main result is a different formulation of this one, that makes sense over $\mathbb{Z}$, however, the formulation we give here is the one that generalizes most naturally to the quantum context.

\section{The Quantum Abelian/non-Abelian Correspondence}
We state here a generalization to quantum $K$-theory of Harada-Landweber's abelian/non-abelian correspondence. To do this, we require one additional hypothesis on $V//G$, that $G$-unstable locus be of codimension two. Under this hypothesis, classical abelian/non-abelian correspondence for cohomology can be used to identify the curve classes on $V//G$ with the Weyl coinvariants of the curve classes on $V//T_G$, giving a natural map from the Novikov variables of $V//T_G$ to those of $V//G$. Extending $\phi$ by this map gives a map $\phi_Q:K(X//T)^W[[Q^i_j]]\to K(X//G)[[Q_i]]$. \\

The following conjecture, introduced in \cite{me}, gives a quantum generalization of the $K$-theoretic abelian/non-abelian correspondence. 
\begin{conj}
\label{abconj}
    $\phi_Q$ is a surjective ring homomorphism from $QK^{tw}(V//T_G)^W$ to $QK(V//G)$

    Where the twisting $tw$ is determined by $Eu_{\lambda}(\bigoplus_r L_r)$, here $Eu_{\lambda}$ is the $\mathbb{C}^*$-equivariant Euler class with respect to the action scaling the fibers.\\

    Furthermore, $\phi_Q$ respects quantum pairings in the following sense:

    $$\frac{1}{|W|}((a,b))^{tw}=((\phi_Q(a),\phi_Q(b)))$$

\end{conj}

In \cite{me}, this conjecture was proven in our case of interest, based on a corresponding result of Yan in \cite{Xiaohanflag} for the range of the big symmetrized $J$-functions:
\begin{thm}
    Conjecture \ref{abconj} is true for $X=Fl(v_1,\dots,v_n;N)$. 
\end{thm}

However, to actually apply this theorem, we need to understand the geometry of $Y$, the abelianization of the flag. 

\section{The Flag Variety and its Abelianization}

The standard description of the type $A$ flag variety is as a quotient of the Lie group $SL(n,\mathbb{C})$ by a parabolic subgroup. However the most useful description for us is a different one, coming from geometric invariant theory.

The partial flag variety $Fl(v_1,\dots,v_n; N)$ is given as the following GIT quotient $V//_\theta G$, where:

\begin{itemize}
    \item $V=\prod_{i=1}^{n-1}Hom(\mathbb{C}^{v_i},\mathbb{C}^{v_{i+1}})\times Hom(\mathbb{C}^{v_n},\mathbb{C}^N)$. It is a $v_1v_2+\dots Nv_{n}$-dimensional space, with coordinates labelled $z^a_{b,c}$, where $a$ denotes the choice of matrix, and $b,c$ denote the entry. It is also acted on by the \textit{large} torus $$\mathbb{T}:=(\mathbb{C}^*)^{v_1v_2+\dots Nv_n}$$
    Whose lie algebra is described by coordinates $x^a_{b,c}$. 

    \item $G=\prod_i GL(v_i)$, where an element $(g_1,\dots, g_n)$ acts on a set of matrices $(M_1,\dots,M_n)$ by sending it to $(g_1M_1g_1^{-1},\dots, M_n g_n^{-1})$.
    \item $\theta$ is the character of $G$ given by taking the determinant on each factor, which defines the stability condition which selects sets of matrices that define injective linear maps. 
\end{itemize}

$$\prod_{i=1}^{n-1}Hom(\mathbb{C}^{v_i},\mathbb{C}^{v_{i+1}})\times Hom(\mathbb{C}^{v_n},\mathbb{C}^N)//_{\theta}\prod_i GL(v_i)$$

We recall some facts about the geometry of $Fl(v_1,\dots,v_n;N)$, which we henceforth denote $X$. 

The bundles $\mathbb{C}^{v_i}$ on $V$ descend to bundles $\mathcal{S}_i$ on $X$, whose fiber over a point is the $i$th vector space in the flag that point represents. We make the convention that $\mathcal{S}_{n+1}=\mathbb{C}^{N}$, the trivial bundle of rank $N$ on $X$. \\

We also define the successive quotient bundles $\mathcal{R}_i$ as $\mathcal{S}_{i}/\mathcal{S}_{i-1}$.
Letting $\Lambda_y(E)$ denote the class $\sum_i y^i \wedge^i E$, the Whitney formula gives us the following relations in $K^*(X)$

\begin{equation}
    \label{whit}
    \Lambda_y(\mathcal{S}_i)\Lambda_y(Q_{i+1})=\Lambda_y(\mathcal{S}_{i+1})
\end{equation}

The ring $K^*(X)$ is generated by the classes $\wedge^j\mathcal{S}_i$ and determined entirely by the relations \eqref{whit}, which we henceforth refer to as the \emph{Whitney relations}. 

Rather than just working $K^*(X)$, we work torus equivariantly. The standard torus action of $T^N$ on $\mathbb{C}^N$ induces an action on $X$. If $\Lambda_i$ represent the standard representations of each factor of $T^N$, then in $K^T(X)$, we have:

$$\mathcal{S}_{n+1}=\mathbb{C}^N=\sum_i \Lambda_i$$

After taking this into account, the Whitney relations \eqref{whit} also give a complete description of $K^T(X)$.

\subsection{The Abelianization of the Flag}
The maximal torus $T_G$ inside $G=\prod_{i=1}^n GL(v_i)$ to be the set of elements where each matrix has no off-diagonal entries, determines a new variety $Y:=V//_\theta T_G$, whose geometry we describe in this section. 

Since the final summand of $V$ is $Hom(\mathbb{C}^{v_n}, \mathbb{C}^N)$, and it is only acted on by the torus in $Gl(v_n)$, $Y$ is fibered over the quotient $Hom(\mathbb{C}^{v_n},\mathbb{C}^N)//Diag(v_n)\cong (\mathbb{C}P^{N-1})^{v_n}$. \\

By a similar argument, we find that $Y$ is a tower of projective bundles, 
$$F_1=Y\to F_2\to F_3\dots \to F_n\to  F_{n+1}=(\mathbb{C}P^{N-1})^{v_n}$$

Where $F_i\to F_{i+1}$ is $(\mathbb{C}P^{v_i-1})^{v_{i-1}}$ bundle.  \\

Let $P^i_j$ denote the $j$th tautological bundle $\mathcal{O}(-1)$ in the fiber of $F_{i-1}$, and $p^i_j:=-c_1(P^i_j)$, then the above description gives us the following facts:

\begin{itemize}
    \item $P^i_j$ generate $K^*(Y)$
    \item The duals to $p^i_j$ are effective curve classes and generate $H_2(Y,\mathbb{Z})$. In fact, their positive span generates the Mori cone. We will use these to index the Novikov variables $\mathbb{Q}^i_j$ for $Y$. 
    \item The Weyl group $W$ is $\prod_i S_{\nu_i}$, and permutes the bundles $P^i_j$. The bundles associated to the simple roots are $\frac{P^i_j}{P^i_k}$. 
    
\end{itemize}

This description also determines the map $\phi$. $\mathcal{S}_i$ and $\bigoplus_j P^i_j$ are both bundles determined from the same $G-$representation, $Gl(v_i)$ acting on $\mathcal{C}^{v_i}$, thus they are related by $\phi$. Similarly, any symmetric function of the Chern roots of $\mathcal{S}_i$ is the image of the same function of the bundles $P_j^i$. This identification determines the entire map $\phi$, and justifies our abuse of notation in the introduction identifying $P_j^i$ with the corresponding Chern roots.\\

The map on Novikov variables is simply $Q^i_j\mapsto Q_i$. 

\subsection{Toric Description}
To obtain further information about the geometry of $Y$, we use the fact that it is a toric variety (this is a consequence of a being a GIT quotient of a linear space by a torus). \\

The action of the \emph{large torus} $\mathbb{T}=(\mathbb{C}^*)^M$ on $V$ induces a corresponding action on $Y$, which gives the standard torus action on a toric variety. This procedure is explained in \cite{givtor}, where the secondary fan is referred as to the \say{picture}. We will describe the results of the procedure here, and then explain what this says about the geometry of $Y$.

From the description of $Y$ as a GIT quotient, we can construct the \emph{secondary fan} of $Y$, inside $H^2(Y,\mathbb{R})$ with an integral basis of $p^i_j$, for $1\leq i\leq n$ and $1 \leq j\leq v_i$. It is generated by the rays $u^i_{j.k}=p^{i}_j-p^{i+1}_k$, each corresponding to a $\mathbb{T}-$invariant divisor. Here we make the convention that $p^{n+1}_j=0$. 

The Kahler cone is the intersection of all maximal cones containing the stability condition (which is realized in $H^2(Y)$ as the first Chern class of the line bundle induced by the character $\theta$, and corresponds to the point $(1,1,\dots,1)$). In this case, it is precisely the positive span of the $p^i_j$. \\

A maximal cone $\sigma$ of the secondary fan has two interpretations, depending on if $\theta\in \sigma$ or not. If $\theta\in \sigma$, the cone determines an isolated fixed point of the $\mathbb{T}$-action (all fixed points come from such cones). \\

Otherwise, the divisors determined by the rays of the cone have empty intersection. Given a cone $\sigma$ denote $r(\sigma)$ the set of rays. 

The linear relations between the $u^i_{j,k}s$ and the relations determined by cones not containing $\theta$, collectively known as the Kirwan relations, determined $H^\mathbb{T}(Y)$. A similar presentation holds for $K^*(Y)$, using the line bundles $P^i_j, U^i_{j,k}$ whose first Chern classes are $-p^i_j, -u^i_{j,k}$ respectively. If $\alpha$ is a ray in the secondary fan, let $U_\alpha$ denote the corresponding line bundle (it will be $U^{i}_{j,k}$ where $\alpha=u^{i}_{j,k})$. Let the equivariant parameters corresponding to the $\mathbb{T}$-action be denoted $\Lambda^i_{j,k}$, then:

$$K^\mathbb{T}(Y)=\mathbb{C}[P^i_j] / \langle U^{i}_{j,k}=\frac{P^{i}_{j}}{\Lambda^{i}_{j,k}P^{i+1}_k},\prod_{\alpha\in r(\sigma)\text{ $\theta\notin \sigma$ } }U_\alpha=0\rangle$$

The final ingredient we need here is how to do fixed-point localization on $Y$. A fixed point is determined by a cone $\sigma$ containing $\theta$. The localization of a class is determined by setting $U_\alpha=1$ for $\alpha\in \sigma$. Doing this determines the images of $P^{i}_{j}$, which determine the images of the other $U_\alpha$.

A maximal cone containing $\theta$ must be of the following form:

For each $1\leq i\leq n-1$, choose some injective function 

$$f_i:\\{1,\dots,v_i\\} \to \\{1,\dots,v_{i+1}\\}.$$

Let $S_i$ denote the set of vectors $p^i_j-p^{i+1}_{f_i(j)}$, with $S_{n}$ corresponding to the set of $p^n_j$. Then for some choice of $f_i$s, any maximal cone has rays given by $\bigcup_i S_i$. These cones are acted on transitively by the Weyl group. 

We choose a distinguished fixed point $\tilde{A}$, corresponding to the cone determined by $f_i(k)=k$.

We can actually simplify this picture. Since $Y$ is a GIT quotient of $\mathbb{C}^M$ by $T_G$, and the action of $\mathbb{T}$ is induced from the action on $\mathbb{C}^M$, we can equivalently consider the action by the quotient $\tilde{T}=\mathbb{T}/T_G$. This is the standard torus in the theory of toric varieties. Restricting to this quotient is equivalent to letting $\Lambda^{i}_{j,k}=\Lambda^i_{s,k}=\Lambda^{i}_{k}$ for all $j,s$. 

From the perspective of matrices, $\Lambda^{i}_{j,k}$ corresponds to scaling the $k,j$th element of the $i$th matrix, and $\Lambda^{i}_{k}$ corresponds to scaling the entire $k$th row. 

Localizing with respect to the $\tilde{T}-$action corresponds to setting $P^i_j/P^{i+1}_{f_i(j)}$ to $\Lambda^{i}_{f_i(j)}$, and specializing other variables accordingly. At the fixed point $\tilde{A}$, this sends $P^i_j/P^{i+1}_j$ to $\Lambda^i_j$. 

The $\tilde{T}$-action specializes to the $T$-action coming from $\mathbb{C}^N$ by sending $\Lambda^{n-1}_{r}\to \Lambda_r$ and $\Lambda^{k}_{j}\to 1$ for all other $k$.

\subsection{Proof Strategy}
With these preliminaries established, our proof strategy is as follows. We first write down a function that is a deformation of the twisted small $J$-function of $Y$ and use it to obtain a presentation for $QK^{tw}(Y)$, with similarly deformed relations. 

From there, we find appropriate Weyl-symmetrizations of relations in $QK^{tw}(Y)$, and compute how they specialize to $X$, after specializing to $X$, the deformations we introduced before will vanish, so our resulting relations are honest relations in $QK(X)$.

This will recover a stronger variant of conjecture \ref{whitconj}. We show that this variant implies the Rimanyi-Varchenko-Tarasov presentation. Subsequently, we use the results to give an isomophism between $QK(Fl)$ and $QK^{Fl}(QM)$.

\section{$QK^{tw}(Y)$}
\subsection{(Almost) The Twisted $J$-Function}
We can now calculate a presentation for $QK^{tw}(Y)$, and use an appropriate symmetrization to obtain the Whitney relations. To do this, we use the following value of the $S_n$-equivariant twisted $J$-function of $Y$, obtained by Yan in \cite{Xiaohan}, which we denote $\tilde{J}_{tw}^Y$ it is described the the following expression:

\begin{remark}
    For the sake of brevity, and since this is the only value we need the argument, we do not define  The reader is invited to refer to \cite{me} or \cite{Xiaohan} for more details. 
\end{remark}

\begin{equation}\label{jf}
\resizebox{\hsize}{!}{$
(1-q)\sum_{d\in \mathcal{D}} \prod_{i,j} (Q^i_{j})^{d^i_{j}}\frac{\big(\prod_{i=1}^{n-1}\prod_{1\leq r\leq v_{i+1}}^{1\leq s\leq v_i}\prod_{l=-\infty}^{0}(1-\frac{P^i_{s}}{\Lambda^{i}_{r}P^{i+1}_r}q^l)\cdot\prod_{1\leq r\leq N}^{1\leq s\leq v_n}\prod_{l=-\infty}^{0}(1-\frac{P^n_{s}}{\Lambda^{n}_{r}q^l})\big)\prod_{i=1}^n\prod_{r\neq s}^{1\leq r,s\leq v_i}\prod_{l=-\infty}^{d^i_{s}-d^i_{r}}(1-\lambda\frac{P^i_{s}}{P^i_{r}}q^l)}{\big(\prod_{i=1}^{n-1}\prod_{1\leq r\leq v_{i+1}}^{1\leq s\leq v_i}\prod_{l=-\infty}^{d^i_{s}-d^{i+1}_r}(1-\frac{P^i_{s}}{\Lambda^{i}_{r}P^{i+1}_r}q^l)\cdot\prod_{1\leq r\leq N}^{1\leq s\leq v_n}\prod_{l=-\infty}^{d^n_{s}}(1-\frac{P^n_{s}}{\Lambda^{n}_{r}q^l})\big)\prod_{i=1}^n\prod_{r\neq s}^{1\leq r,s\leq v_i}\prod_{l=-\infty}^{0}(1-\lambda\frac{P^i_{s}}{P^i_{r}}q^l)}$}
\end{equation}
 We note that this expression is rather cumbersome. If we simplify the notation by defining the modified product $$\widetilde{\prod}_{i=1}^k f_k:=\frac{\prod_{i=-\infty}^k f_k}{\prod_{i=-\infty}^0 f_k}$$

    Then 

    $$\tilde{J}_Y^{tw}=(1-q)\sum_{d\in \mathcal{D}} \prod_{i,j} (Q^i_{j})^{d^i_{j}}\frac{\prod_{i=1}^n\prod_{r\neq s}^{1\leq r,s\leq v_i}\widetilde{\prod}_{l=1}^{d^i_{s}-d^i_{r}}(1-\lambda\frac{P^i_{s}}{P^i_{r}}q^l)}{\prod_{i=1}^{n-1}\prod_{1\leq r\leq v_{i+1}}^{1\leq s\leq v_i}\widetilde{\prod}_{l=1}^{d^i_{s}-d^{i+1}_r}(1-\frac{P^i_{s}}{\Lambda^{i}_{r}P^{i+1}_r}q^l)\cdot\prod_{1\leq r\leq N}^{1\leq s\leq v_n}\widetilde{\prod}_{l=1}^{d^n_{s}}(1-\frac{P^n_{s}}{\Lambda^{n}_{r}q^l})}$$

 Since we are only interested in the $T$-equivariant theory, rather than the $\tilde{T}$-equivariant theory, we make the specialization $\Lambda^i_j\to 1$ for $i<n$, and $\Lambda^n_j\to \Lambda_j$. Call the resulting function $\overline{J}^{tw}_Y$, we have:

\begin{equation}\label{js1}
\resizebox{\hsize}{!}{$
\overline{J}^{tw}_Y=(1-q)\sum_{d\in \mathcal{D}} \prod_{i,j} (Q^i_{j})^{d^i_{j}}\frac{\big(\prod_{i=1}^{n-1}\prod_{1\leq r\leq v_{i+1}}^{1\leq s\leq v_i}\prod_{l=-\infty}^{0}(1-\frac{P^i_{s}}{P^{i+1}_r}q^l)\cdot\prod_{1\leq r\leq N}^{1\leq s\leq v_n}\prod_{l=-\infty}^{0}(1-\frac{P^n_{s}}{\Lambda_{r}q^l})\big)\prod_{i=1}^n\prod_{r\neq s}^{1\leq r,s\leq v_i}\prod_{l=-\infty}^{d^i_{s}-d^i_{r}}(1-\lambda\frac{P^i_{s}}{P^i_{r}}q^l)}{\big(\prod_{i=1}^{n-1}\prod_{1\leq r\leq v_{i+1}}^{1\leq s\leq v_i}\prod_{l=-\infty}^{d^i_{s}-d^{i+1}_r}(1-\frac{P^i_{s}}{P^{i+1}_r}q^l)\cdot\prod_{1\leq r\leq N}^{1\leq s\leq v_n}\prod_{l=-\infty}^{d^n_{s}}(1-\frac{P^n_{s}}{\Lambda_{r}q^l})\big)\prod_{i=1}^n\prod_{r\neq s}^{1\leq r,s\leq v_i}\prod_{l=-\infty}^{0}(1-\lambda\frac{P^i_{s}}{P^i_{r}}q^l)}$}
\end{equation}

$\overline{J}^{tw}_Y$ satisfies the following system of $q-$difference equations, one for each appropriate value of $i,j$. We make the convention that $P^n_j=\Lambda_j$, $v_n=N$, $v_0=0$.

$$\prod_{k\neq j} (1-\Lambda q\frac{P^i_k}{P^i_j}q^{Q^i_k\partial_{Q^i_k}-Q^i_j\partial_{Q^i_j}})\prod_k(1-\frac{P^i_j}{P^{i+1}_k}q^{Q^i_j\partial_{Q^i_j}-Q^{i+1}_k\partial_{Q^{i+1}_k}})\overline{J}^{tw}_Y=$$
$$Q^i_j\prod_k(1-\frac{P^{i-1}_k}{P^{i}_j}q^{Q^{io1}_k\partial_{Q^{i-1}_k}-Q^i_j\partial_{Q^i_j}})\prod_{k\neq j}(1-\Lambda q\frac{P_j^i}{P_k^{i}}q^{Q^i_j\partial_{Q^i_j}-Q^i_k\partial_{Q^i_k}})\overline{J}^{tw}_Y$$

If $\overline{J}^{tw}_Y$ were the small twisted $J$-function of $Y$, the symbols of these operators would be relations in $QK^{tw}(Y)$, hence $W$-invariant combinations of these relations would descend to $QK(X)$. The degree bounds proved previously would imply $\hat{P}_i=P_i$. 

Although this is not in fact the case, all is not lost. By \cite[Thm 6.15]{me}, since $\phi(\overline{J}^{tw}_Y)=J_X$, applying the above procedure to $\overline{J}^{tw}_Y$ still results in relations in $QK(X)$, essentially because the symbols of the operators do not literally give relations in $QK^{tw}(Y)$, but they define relations up to terms lying in a submodule $\mathscr{I}$ (the submodule of $K(Y)[[Q]]$ generated by $\ker(\phi)$) that vanish after the application of $\phi$. 

The quantum triviality theorem also applies after replacing $J^{tw}_Y$ with $\overline{J}^{tw}_Y$ up to $\mathscr{I}$ (see \cite[Cor 4.14]{me}), but the resulting quantum deformations are trivial up to elements of $\mathscr{I}$.

Thus the symbols of the operators translate to the following \say{relations} in $QK^{tw}(Y)$, which in reality only hold up to elements of $\mathscr{I}$.

\begin{equation}
\label{eqq}
\prod_{b=1}^{v_{n+1}}\prod_k(1-\frac{P^i_j}{P^{i+1}_b})=Q_{ij}\prod_{a=1}^{v_{n-1}}(1-\frac{P^{i-1}_a}{P^{i}_j})\prod_{k\neq j}\frac{(1-\Lambda \frac{P_j^i}{P_k^{i}})}{ (1-\Lambda \frac{P^i_k}{P^i_j})}
\end{equation}

In the special cases $i=n$, $i=1$, \eqref{eqq} becomes:

\begin{equation}
\prod_{b=1}^{N}\prod_k(1-\frac{P^n_j}{\Lambda_b})=Q_{nj}\prod_{a=1}^{v_{n-1}}(1-\frac{P^{n-1}_a}{P^{n}_j})\prod_{k\neq j}\frac{(1-\Lambda \frac{P_j^n}{P_k^{n}})}{ (1-\Lambda \frac{P^n_k}{P^n_j})}
\end{equation}

\begin{equation}
\prod_{b=1}^{v_2}\prod_k(1-\frac{P^1_1}{P^2_b})=Q_{nj}
\end{equation}


\subsection{Degree Bounds}

Before proceeding by taking symmetric combinations of the above relations, we establish some asymptotic information:

First write:

$$\tilde{J}^{tw}_Y=(1-q)\sum_d Q^d J_d$$
We have:

\begin{equation}
\label{degree}
 deg(J_d)= \deg(J_d|_A)=\sum_{i=1}^n \Bigg(\sum_{j,k\leq v_i} \binom{d_{ij}-d_{ik}+1}{2}-\sum_{q\leq v_i,r\leq v_{i+1}}\binom{d_{iq}-d_{i+1,r}+1}{2}\Bigg)
\\
\end{equation}

 \text{Here $\binom{a}{2}=0$ if $a\leq 0$, and $d^{n+1}_k=0$ for any $k$.}

This allows us to conclude:

\begin{lem}
$\tilde{J}^{tw}_Y$ is equal to $1-q$, up to terms vanishing at $q=\infty$
\end{lem}

Without loss of generality, we can also assume that for each $i$, the integers $d^i_j$ are in increasing order. Having done so, we can rewrite the right hand side of \eqref{degree} as:

 $$\sum_{i} \Bigg(\big(\sum_{j,k\leq v_i} \binom{d_{ij}-d_{ik}+1}{2}-\sum_{q\leq v_i,r\leq v_{i}}\binom{d_{iq}-d_{i+1,r}+1}{2}\big)-\sum_{q\leq v_i, v_i<r\leq v_{i+1}} \binom{d_{iq}-d_{i+1,r}+1}{2}\Bigg)$$

  Define $$T_i:=\big(\sum_{j,k\leq v_i} \binom{d_{ij}-d_{ik}+1}{2}-\sum_{q\leq v_i,r\leq v_{i}}\binom{d_{iq}-d_{i+1,r}+1}{2}\big)$$
  
  and 
  $$R_i:=\sum_{q\leq v_i, v_i<r\leq v_{i+1}} \binom{d_{iq}-d_{i+1,r}+1}{2}$$
  
  $T_i\leq 0$ because whenever $d_{ij}-d_{ik}>0$, then $d_{ij}-d_{ik}\leq d_{ij}-d_{i+1,k}$.

Furthermore, $T_{n}-R_n=\sum_{i.j}\binom{d^n_i-d^n_j+2}{2}-N\sum_{i}\binom{d^n_i+1}{2}$. 

The positive part is strictly less than $v_n\sum_{i}\binom{d^n_i+1}{2}$, so unless all $d^n_j=0$, we have:

$$deg(J_d)\leq T_n-R_n< v_n-N\leq -1$$

If all the $d^n_j$ are 0, we can apply the same argument to show that $T_{n-1}<v_{n-1}-v_n$ unless all the $d^{n-1}_k$s are 0. Continuing this procedure shows that unless $d=0$: 

$$deg(J_d)<-1$$

Thus the rational function $(1-q)J_d$ has no poles at $q=\infty$, furthermore, it vanishes there.

In fact, we can give stronger bounds on the degree of $J_d$, which we will need for the purposes for using the quantum triviality theorem:

\begin{lem}
For any $i$,
$$deg(J_d)<-\sum_j d^i_{j}$$
\end{lem}

\begin{proof}
As before, write $$deg(J_d)= \sum_{i=1}^{n} T_i-R_i$$. For simplicity, we assume not all $d^n_j=0$ (if not, this is the degree of $J_d$ inside a smaller flag). 

Let $S_k=\sum_{i=n-k}^{n} T_i-R_i$. Since the $S_k$s are decreasing and $S_n=deg(J_d)$, our result is implied by the claim that $S_k\leq \sum_j d^{n-k}_j $

We will induct on $k$. The base case $k=0$, $S_k=T_n-R_n$.

We have already established earlier that
$$T_n-R_n<(v_n-N)\sum_j d^n_j,$$ so the inequality is satisfied in this case. 

For the induction step, let $i=n-k$, so $S_{k}=T_i+S_{k-1}$. 

We first observe that $$T_i-R_i\leq T_i= \sum_{j,k\leq v_{i}} d^{i}_j-d^i_k-d^i_j-d^{i+1}_k=v_i\big(\sum_{j=1}^{v_i} (d^{i+1}_j-d^i_j)\big)$$

By the induction hypothesis $S_{k-1}<-\sum_{j=1}^{v_{i+1}} d^{i+1}_j$. If $\sum_j d^{i}_j\leq\sum_j d^{i+1}_j$, the induction hypothesis implies our desired result. However, if this is not the case, then $\sum_{j=1}^{v_{i+1}} d^{i+1}_j=\sum_{j=1}^{v_{i}} d^{i+1}_j+C$, where $C>\sum_{j=1}^{v_i} (d^{i}_j-d^{i+1}_j)$. 

Thus we have:

$$S_{k}<v_i\big(\sum_{j=1}^{v_i} (d^{i+1}_j-d^i_j)\big)+S_{k-1}\leq v_i-1\sum_{j=1}^{v_i} (d^{i+1}_j-v_i\sum_{j=1}^{v_i}d^i_j)-C<v_i\big(\sum_{j=1}^{v_i} (d^{i+1}_j-d^i_j)\big)-\sum_{j=1}^{v_i}d^{i}_j\leq -\sum_{j=1}^{v_i} d^{i}_j$$

\end{proof}

The second bound we need to establish is:

\begin{lem}
For an integer $0\leq \ell \leq v_{i+1}-v_i$, let $F_\ell(d^{i})$ denote the maximum sum of $\ell$ distinct $d^i_j$s.

For any choice of $i,j$, we have: 

$$deg(J_d)< -F_\ell(d^{i+1})+(v_{i}-v_{i+1}+\ell)d^i_j$$

\end{lem}

\begin{proof}

We first address the case $\ell=0$, corresponding to the inequality:

$$deg(J_d)< (v_{i}-v_{i+1})max_j(d^i_j)$$

As before, we assume without loss of generality that $d^n_j$ are not all equal to 0.

Let $b_{i}=max_j(d^{i}_j)$. 

We have that for all $j$, 

\begin{equation}
    \label{master}
    \sum_j sum_{k=v_{i+1}-v_i}^{v_{i+1}} max(0,d^{i}_j-d^{i+1}_k)\leq R_i. 
\end{equation}

Thus, choosing the $j$ which maximizes the value of $d^i_j$, recalling that $T_i\geq 0$:
\begin{equation}
    \label{grn1}
    T_i-R_i\leq \sum_{k=(v_{i}-v_{i+1})}^{v_{i+1}}min(0,d^{i+1}_k-d^i_j)
\end{equation}

We can further obtain the following bound on the next term in the expression for the degree:

\begin{equation}
    \label{grn2}
    T_{i+1}-R_{i+1}\leq \sum_{k=(v_{i}-v_{i+1})}^{v_{i+1}}d^{i+2}_k-d^{i+1}_k
\end{equation}

This inequality holds because we can extract terms $-\binom{d^{i+1}_k-d^{i+2}_k+1}{2}$ from $T_{i+1}$.

The same inequality holds for $i+s$ for arbitrary $s$. In addition, we note that it is necessarily strict at some point (eventually all the $d^{i+s}_k$ must vanish).

For a given $k$, the contribution to the sum of all of these inequalities is the following:

$$min(0,d^{i+1}_k-d^i_j)+(d^{i+1}_k-d^{i+2}_k)+\dots d^n_k$$

If we eliminate the $min(0,\cdot)$s from the sum, the right hand side telescopes to $(v_i-v_{i+1})d^i_j$, yielding the desired inequality. In fact, the same inequality holds with the $min(0,)$ present. The minimum being achieved at $0$ of a given term corresponds to $d^{s+1}_{k+v_{s+1}-v_{i+1}}\geq d^{s}_{k+v_s-v_{i+1}}$. 

This means that so long as at least one of the terms in the sum for a fixed $k$ is nonzero, which is guaranteed since $d^{n+1}_k=0$ for all $k$, the sum of those terms telescopes to a quantity that must be at most $-d^i_j$, proving the desired inequality.












For general $\ell$. If $F_\ell(d^{i+1})+(v_{i+1}-v_i-\ell)max_j(d^i_j)$ is smaller than the corresponding term for $\ell=0$, the same bound applies. If it is larger, that means the largest $\ell$ $d^{i+1}_k$s total to more than $\ell max_j(d^i_j)$. We can replace $\ell$ of the $d^i_j$s with these $d^{i+1}_k$s and run the same argument.

\end{proof}

\section{Characteristic Polynomials and Symmetrization}

Using the ring-theoretic abelian/non-abelian correspondence and \cite[Thm 6.15]{me}, we can obtain relations in $QK(X)$ via taking $W-$invariant combinations of \eqref{eqq}, mapping $Q^{i}_j$ to $Q_i$, taking the limit $y\mapsto 1$, evaluating all the quantum products, and then specializing via the classical abelian/non-abelian correspondence. The fact that the relations we found are true up to elements of $\mathscr{I}$ makes no difference, since $W$-invariant combinations of those elements are sent to $0$ under specialization.

We can simplify this procedure in the following way. Rather than finding $W$-symmetrizations of \eqref{eqq}, and taking their specializations, we will specialize $Q$s and $y$ first, and find $W$-symmetrizations of the resulting expressions. The resulting specialized relations are:

\begin{equation}
\prod_{b=1}^{v_{n+1}}\prod_k(1-\frac{P^i_j}{P^{i+1}_b})=Q_{i}\prod_{a=1}^{v_{n-1}}(1-\frac{P^{i-1}_a}{P^{i}_j})\prod_{k\neq j}\frac{(1- \frac{P_j^i}{P_k^{i}})}{ (1- \frac{P^i_k}{P^i_j})}
\end{equation}

Noting that $\frac{1-\frac{x}{y}}{1-\frac{y}{x}}=\frac{-x}{y}$, we can rewrite this equation as:

\begin{equation}
\label{Bethe}
(-1)^{v_i-1}\prod_k P^i_k\prod_{b=1}^{v_{n+1}}\prod_k(1-\frac{P^i_j}{P^{i+1}_b})=(P^i_j)^{v_i}Q_{i}\prod_{a=1}^{v_{n-1}}(1-\frac{P^{i-1}_a}{P^{i}_j})
\end{equation}

This equation is both equivalent to the Bethe Ansatz from the 5-vertex lattice model, and the Coulomb branch equation coming from 3D GLSM introduced in \cite{conj} that conjecturally describes the quantum $K$-theory of $X$. 
\subsection{Whitney Presentation}
The Whitney presentation was conjectured in \cite{conjpf} based on applying certain algebraic manipulations to symmetrize \eqref{Bethe}, regarded as the Coulomb branch equations from a GLSM. In essence, our work gives precise mathematical meaning to these manipulations, in terms of the abelian/non-abelian correspondence, and, with a bit of extra work, converts the ideas of \cite{conjpf} into a proof of the Whitney presentation. 

The symmetrization procedure below is identical to the one in \cite{conjpf}. The new ingredients are the interpretation of the equations \eqref{Bethe} as relations in $QK^{tw}(Y)$ up to an ideal $\mathscr{I}$, and Lemma 8.1, which appears later. 

Using $P^i$ as shorthand for the collection of variables $P^i_j$, and $e_i$ denoting the $i$th elementary symmetric polynomial for $i\geq 0$, and 0 for $i<0$: 
If we define the polynomial $F_i(t)$ as:

$$\sum_{\ell=0}^{v_{i+1}} t^\ell(e_{v_i}(P^i)e_\ell(P^{i+1})+Q_ie_{v_i+1}(P^{i+1})e_{\ell-v_{i+1}+v_i}(P^{i-1}))$$

Then the equations \eqref{Bethe} are equivalent to:
\begin{equation}
    F_i(P^i_j)=0
\end{equation}

We will use Vieta's formulas for $F_i$ to obtain $W$-symmetrizations of this relation. $F_i$ has $v_{i+1}$ roots, $v_i$ of which are $P^i_j$s, we denote the remaining roots by the set $\bar{P}^i$, and the whole set of roots by $w$. 
We thus have: 
\begin{equation}
\label{a}
e_\ell(w)=\sum_{i=0}^{n-k} e_{\ell-i}(P^i)e_{i}(\bar{P}^i) \end{equation}

Applying Vieta's formula to $F$ gives:

\begin{equation}
\label{b} e_{v_i}(P^i)e_\ell(w)=e_{v_i}(P^i)e_{\ell}(P^{i+1})+Q_ie_{v_{i+1}}(P^{i+1})e_{\ell-v_{i+1}-v_i}(P^{i-1})\end{equation}

Looking at \eqref{a} for $\ell=v_{i+1}$ gives a way to eliminate $e_{v_{i+1}}(P^{i+1})$:

$$e_{v_{i+1}}(w)=e_{v_{i+1}}(P^{i+1})=e_{v_i}(P^i)e_{v_{i+1}-v_i}(\bar{P}^i)$$

Applying this to \eqref{b} gives:

\begin{equation}
\label{c}e_\ell(w)=e_{\ell}(P^{i+1})+Q_ie_{v_{i+1}-v_i}(\bar{P}^{i})e_{\ell-v_{i+1}-v_i}(P^{i-1})\end{equation}

Substituting \eqref{a} yields:

\begin{equation}
\label{c'}\sum_{j=0}^{v_{i+1}-v_i}e_{\ell-j}(P^i)e_{j}(\bar{P}^i)=e_{\ell}(P^{i+1})+Q_ie_{v_{i+1}-v_i}(\bar{P}^{i})e_{\ell-v_{i+1}-v_i}(P^{i-1})\end{equation}

Continuing to follow \cite{conj}, \eqref{c'} is the degree $\ell$ part of the following product:

\begin{equation}
\label{d}\sum_{j=0}^{v_i}y^je_{j}(P^i)\sum_{k=0}^{v_{i+1}-v_i}y^ke_{k}(\bar{P}^i)=(\sum_{j=0}^{v_{i+1}}y^je_{j}(P^{i+1}))+Q_iy^{v_{i+1}-v_i}e_{v_{i+1}-v_i}(\bar{P}^{i})\sum_{k=0}^{v_{i-1}}e_k(P^{i-1})\end{equation}

We can solve \eqref{d} for the generating function of $e_\ell(\bar{P}^i$ to yield:

\begin{equation}
\label{e}\sum_{k=0}^{v_{i+1}-v_i}y^ke_{k}(\bar{P}^i)=(\sum_{r}(-y)^rh_r(P^i))\times ((\sum_{j=0}^{v_{i+1}}y^je_{j}(P^{i+1}))+Q_iy^{v_{i+1}-v_i}e_{v_{i+1}-v_i}(\bar{P}^{i})\sum_{k=0}^{\ell-v_{i+1}-v_i}e_k(P^{i-1}))\end{equation}

This gives the following description of $e_\ell(\bar{P}^i):$

\begin{equation}
 \label{f}
 e_\ell(\bar{P}^i)=\begin{cases} \sum_{j=0}^{v_{i+1}}(-1)^j e_{\ell-j}(P^{i+1})h_{j}(P^i) & \ell<v_{i+1}-v_i\\ (1-Q_i)^{-1}\sum_{j=0}^{v_{i+1}}(-1)^j e_{\ell-j}(P^{i+1})h_{j}(P^i) & \ell=v_{i+1}-v_i\end{cases}
\end{equation}

\begin{lem}
    The quantum products $e_{\ell-j}(P^{i+1})h_{j}(P^i)$ and $e_{\ell}(P^i)$ in $QK^{tw}(Y)$ are equal (up to $\mathscr{I}$) to their classical counterparts.
\end{lem}

\begin{proof}
    This is a direct consequence of the quantum triviality theorem, the operator $e_{\ell-j}(P^{i+1} q^{Q\partial Q})h_j(P^iq^{Q\partial Q})$, when applied to $J_d$, increases the $q-$degree by at most $(v_{i+1}-v_i)max(d^{i}_j)$. So the resulting term vanishes at $\infty$, so the quantum product contains no Novikov variables by the quantum triviality theorem.\\

    The same argument, using the other degree bound, proves the second statement. 
\end{proof}

 It will be convenient to revert to using the notation $\mathcal{R}_i=\frac{\mathcal{S}_{i+1}}{\mathcal{S}_i}$.

Since all of these relations are true up to $W$-invariant elements of $\mathcal{I}$, the specialization map to $QK(X)$ kills the $\mathcal{I}$-deformation, and sends $e_\ell(P^i)$ to $\wedge^\ell \mathcal{S}_i$ and $e_\ell(\bar{P}^i)$ to $(1-Q_i)^{-\delta_{\ell,v_{i+1}-v_i}}\wedge^\ell \mathcal{R}_i$. We refer to this quantity as $\wedge^{\ell}\widehat{\mathcal{R}}_i$, and we call $\widehat{\mathcal{R}}_i$ the \emph{quantum quotient bundle}, for reasons that will be clear later. It is equivalent to what is denoted $\tilde{R}$ in \cite{conj}. 

This means that the specialization of \eqref{e} is equivalent to:

\begin{equation}
\label{gamer}
\Lambda_y(\mathcal{S}_i)\Lambda_y(\widehat{\mathcal{R}_i})=\Lambda_y(\mathcal{S}_{i+1})+Q_idet(\widehat{\mathcal{R}_i})\Lambda_y(\mathcal{S}_{i-1})
\end{equation}

If we eliminate the $\widehat{\mathcal{R}}_i$s, we can rewrite the above equation as:
\begin{equation}
\sum_{r=0}^{v_{i+1}-v_i} \wedge^{\ell-r} \mathcal{S}_i \,* \, 
\wedge^r( \mathcal{R}_i) 
\: = \:
\wedge^{\ell} \mathcal{S}_{i+1} \: - \:
\frac{Q_i}{1-Q_i} \det( \mathcal{R}_i) \, * \,
\left( \wedge^{\ell -v_{i+1}+v_i} \mathcal{S}_i \: - \:
\wedge^{\ell -v_{i+1}+v_i} \mathcal{S}_{i-1}\right).
\end{equation}
Adding factors of $y$, this yields the Whitney relations:
\begin{equation}  \label{whitney2}
\Lambda_y(\mathcal{S}_i) * \Lambda_y( \mathcal{S}_{i+1}/\mathcal{S}_i)
\: = \:
\Lambda_y( \mathcal{S}_{i+1}) \: - \:
y^{v_{i+1}-v_i} \frac{Q_i}{1-Q_i} \det(\mathcal{S}_{i+1}/\mathcal{S}_i) *
\left( \Lambda_y(\mathcal{S}_i) - \Lambda_y(\mathcal{S}_{i-1}) \right).
\end{equation}

\subsection{The Rimanyi-Tarasov-Varchenko Presentation}

Introduce the variables $\gamma^i_j$, $i=0,\dots,n$, $j=1,\dots,dim(\mathcal{R}_i)$. (Here $\mathcal{R}_0$ denotes $\mathcal{S}_1$).

Rimanyi-Tarasov-Varchenko, based on their conjectural isomorphism between $QK(T^*Fl)$ and the Bethe algebra of the 6-vertex model, conjecture the following presentation for $QK(Fl)$. 

Let $$W_Q$$ be an $(n+1)\times (n+1)$ matrix which looks like:

$$\begin{bmatrix}
                \prod_j (1-y\gamma^0_j) & -y^{v_1}\prod_j \gamma^0_j                         &                            &                       &         \\
                Q_1   & \prod_j (1-y\gamma^1_j)      & -y^{v_2-v_1}\prod_j \gamma^1_j                        &                      &           \\                            & \ddots                & \ddots & \ddots & \\                            &                 & Q_{n-1} & \prod_j (1-y\gamma^{n-1}_j)  & -y^{v_{n}-v_{n-1}}\prod_j \gamma^{n-1}_j        \\                        &                      &   &  Q_{n}   & \prod_j (1-y\gamma^n_j)
                \end{bmatrix}$$

Let $Sym(\gamma)$ denote the invariants of $\mathbb{C}[\gamma^i_j]$ under the symmetric group $\prod_{i=0}^n S_{dim(\mathcal{R}_i}$. 

Conjecture 13.17 in \cite{rimanyi2015trigonometric}, states that:

\begin{conj}
    $QK(Fl)\cong Sym(\gamma)/\det(W_Q)=\Lambda_y(\mathbb{C}^N)$
\end{conj}

We prove this conjecture by identifying the $\gamma^i_j$s as \say{Chern roots} of $\widehat{Q}_i$, i.e. they are the roots $\bar{P}_i$ of the Coulomb branch/Bethe Ansatz equation. Thus translated, $Sym(\gamma)$ becomes $\mathbb{C}[\wedge^\ell \widehat{\mathcal{R}}_i]$, and the matrix $W_Q$ becomes:

$$M:=\begin{bmatrix}
                \Lambda_y(\widehat{\mathcal{R}}_0) & -y^{v_1}\det(\widehat{\mathcal{R}}_0)                         &                            &                       &         \\
                Q_1   & \Lambda_y(\widehat{\mathcal{R}}_1)      & -y^{v_2-v_1}\det(\widehat{\mathcal{R}}_1)                 &                      &           \\                            & \ddots                & \ddots & \ddots & \\                            &                 & Q_{n-1} & \Lambda_y(\widehat{\mathcal{R}}_{n-1})  & -y^{v_{n}-v_{n-1}}\det(\widehat{\mathcal{R}}_{n-1})\\                        &                      &   &  Q_{n}   & \Lambda_y(\widehat{\mathcal{R}}_{n})
                \end{bmatrix}$$

We prove a slightly stronger result, which implies this conjecture. 
Let $M_{j}$ denote the submatrix of $M$ consisting of the first $j$ rows and columns.

\begin{thm}
    $\det(M_j)=\Lambda_y(\mathcal{S}_j)$
\end{thm}

\begin{proof}
We induct on $j$. The base case is $M_{1,1}=\Lambda_y(\mathcal{R}_0)=\Lambda_y(\mathcal{S}_1)$. 

For the induction step, we expand along the bottom row of $M_j$, which has two entries: $Q_{j-1}$ and $\Lambda_y(\mathcal{R}_j)$. The minor corresponding to $Q_{j-1}$ has 1 entry in its rightmost column, which is $y^{v_{j}-v_{j-1}}\det(\mathcal{R}_j)$. Thus we calculate the determinant of that minor by expanding along the rightmost column, yielding the following equation:

\begin{equation} 
\det(M_j)=\Lambda_y(\mathcal{R}_j)\det(M_{j-1})-Q_{j-1}y^{v_{j}-v_{j-1}}\det(\mathcal{R}_j)\det(M_{j-2})
\end{equation}

By the induction hypothesis, this relation becomes:

\begin{equation} 
\det(M_j)=\Lambda_y(\mathcal{R}_j)\Lambda_y(\mathcal{S}_{j-1})-Q_{j}y^{v_{j}-v_{j-1}}\det(\mathcal{R}_j)\Lambda_y(\mathcal{S}_{j-2})
\end{equation}

Thus by equation \eqref{gamer}, $\det(M_j)=\Lambda_y(\mathcal{S}_i)$, proving the Rimanyi-Varchenko-Tarasov relations are valid in $QK(Fl)$.

These relations give a complete presentation of the ring by similar arguments to the Whitney presentation (using the $\widehat{\mathcal{R}}$s rather than the $\mathcal{S}$s). 
\end{proof}

\subsection{PSZ Quasimap Rings and the Bethe Ansatz}

It is now time to compare the results with the description of the PSZ quasimap ring of partial flags given in \cite{manybody}. Their description is based on fixed-point localization, rather than in terms of generators and relations. 

Their description is a quantum deformation of the following classical result. The restriction of a $W$-invariant function $\tau(P^i_j)$ to a fixed point $T$ is determined by setting each $P^n_j$ to some $\Lambda_{f(j)}$, and then setting each $P^{n-1}_j$ to another choice of equivariant parameter among the $\Lambda_{f(j)}s$ chosen previously, and continuing down to $P^1_j$. 

Equivalently we specialize each $P^i_j$ to some root of the below equation, such that all choices are unique:

\begin{equation}
    \label{classic}
    \prod_k P^i_j-P^{i+1}_k=0
\end{equation}

This polynomial is of degree $v_{i+1}$, however the $P^i_j$ only correspond to $v_{i}$ of the chosen roots. After restricting to a given fixed point, the remaining $v_{i+1}-v_i$ roots correspond to Chern roots of the quotient bundle $\mathcal{R}_i$.

We can use these perspectives to determine globally valid relations in $K^*(X)$ in the following way:

Consider the polynomials $G_i:=\prod_k(t-P^{i+1}_k)$. 
Any relation among symmetric functions of $P^i_j$ obtained from Vieta's formulas applied to $G_i$ is a globally valid relation in $K_T^*(X)$, since this relation is valid when restricting to any fixed point.

The description of the quantum tautological bundles in \cite{manybody} is essentially the same as the above, except involving the quantum tautological bundles $\widehat{\tau}$. The restriction of $\widehat{\tau}$ to a fixed point is given by evaluating $\tau$ at solutions to \eqref{Bethe}. Thus by a similar argument, relations obtained from applying Vieta's formulas to $F_i(t)$ hold in the quasimap ring, provided we replace $\tau(P^i_j)$ with $\widehat{\tau}(P^i_j)$.

In particular, in there is a ring homomorphism:
$$QK(Fl)\to \mathbb{C}[\Lambda_i][[Q]][\widehat{\tau}]/\langle \text{Bethe Algebra relations} \rangle\text{ given by } \phi_Q(\tau(P^i_j))\mapsto \widehat{\tau}$$

Modulo $Q$, this map is an isomorphism, identifying both rings with $K^*(X)$, thus, by the quantum $K$-theoretic Nakayama lemma, the rings themselves are isomorphic.

\begin{remark}
    The argument above shows that the rings are abstractly isomorphic, it does not show that they give exactly the same product on $K_T^*(X)[[Q]]$, these products could in principle differ by an automorphism of the form $I+o(Q)$. 
\end{remark}

\clearpage

\bibliographystyle{plain} 
\bibliography{cited.bib}

\end{document}